\documentclass[11pt]{article}
\usepackage{amsfonts,latexsym,rawfonts,amsmath,amssymb,amsthm,graphicx}
\numberwithin{equation}{section}
\newtheorem{theorem}{Theorem}[section]
\newtheorem{lem}[theorem]{Lemma}
\newtheorem{thm}[theorem]{Theorem}

\def\s{\,\,\,\,}

\title{Gap theorems for complete submanifolds in the hyperbolic space}
\author{Jianling Liu and Yong Luo}
\date{}
\begin{document} 
\maketitle
\begin{abstract}
Based on the seminal Simons' formula, Shen \cite{Shen} and Lin-Xia \cite{LX} obtained gap theorems for compact minimal submanifolds in the unit sphere in the late 1980's. Then due to the effect of Xu \cite{Xu}, Ni \cite{Ni}, Yun \cite{Yun} and Xu-Gu \cite{XuG}, we achieved a comprehensive understanding of gap phenomena of complete submanifolds with parallel mean curvature vector field in the sphere or in the Euclidean space.  But such kind of results in case of the hyperbolic space were obtained by Wang-Xia \cite{XiaW}, Lin-Wang \cite{LW} and Xu-Xu \cite{XX} until relatively recently and are not quite complete so far. 

In this paper first we continue to study  gap theorems for complete submanifolds with parallel mean curvature vector field in the hyperbolic space, which generalize or extend several results in the literature. Second we prove a gap theorem for complete hypersurfaces with constant scalar curvature $n(1-n)$ in the hyperbolic space, which extends related results due to Bai-Luo \cite{BL2} in cases of the Euclidean space and the unit sphere.  Such kind of results in case of the hyperbolic space are more complicated, due to some extra bad terms in the Simons' formula, and one of main ingredients of our proofs is an estimate for the first eigenvalue of complete submanifolds in the hyperbolic space obtained by Lin \cite{Lin}.
\end{abstract}
\section{Introduction}
An important subject in global differential geometry is the study of rigidity theorems for complete submanifolds in the real space form. In the late 1960’s, Simons obtained the seminal Simons' formula and then Simons \cite{Sim}, Lawson \cite{Law} and Chern-do Carmo-Kobayashi \cite{CDK} used it to verify the famous rigidity theorem for compact minimal submanifolds in the unit sphere via pointwise pinching condition on the squared length of the second fundamental form.

Later on, mathematicians found that Simons' formula is also useful in deriving rigidity theorems for compelete submanifolds with parallel mean curvature vector field under integral curvature conditions. In the late 1980’s, Shen \cite{Shen} studied the rigidity problem for $n$-dimensional minimal hypersurfaces with non-negative Ricci curvature in the unit sphere via integral pinching condition on the $L^n$ norm of the length of the second fundamental form. In \cite{LX}, Lin–Xia obtained an $L^n$-pinching theorem for $2n$-dimensional minimal submanifolds in the unit sphere provided the Euler characteristic not greater than two. More generally, the gap problem for submanifolds in the unit sphere or in the Euclidean space  with parallel mean curvature vector field was investigated by Xu \cite{Xu} and Xu–Gu \cite{XuG}, respectively. Their results can be summarized as follows.
\begin{thm}[\cite{Xu, XuG}]
Let $M^n(n\geq 3)$ be a complete submanifold with parallel mean curvature vector field in the unit sphere or the Euclidean space. Denote by $Å$ the trace free second fundamental form of $M$. There exists an explicit positive constant $C_1(n)$ depending only on $n$,
such that if
$$\int_M|Å|^n \text{d}v<C_1(n),$$
then $Å=0$, i.e. $M$ is totally umbilical. In particular, if $M$ is a minimal submanifold in the Euclidean space, then it is totally geodesic.
\end{thm}
Note that the latter case extended Ni’s \cite{Ni} and Yun's \cite{Yun} gap theorem for minimal hypersurfaces in the Euclidean space. In \cite{XiaW}, Xia-Wang studied the integral rigidity problem for complete minimal submanifolds in the hyperbolic space. They proved 
\begin{thm} [\cite{XiaW}]
Let $M$ be an $n(\geq 5)$-dimensional complete minimal submanifold in the hyperbolic space $\mathbb{H}^{n+m}$. The second fundamental
form $A$ of $M$ satisfying
$$\limsup_{R\to\infty}\frac{\int_{B_R(q)}|A|^2\text{d}v}{R^2}=0,$$
where $B_R(q)$ denotes the geodesic ball on $M$ of radius $R$ centered at $q \in M$. If
$$\int_M|A|^n \text{d}v<C_2(n,m),$$
then $M$ is totally geodesic. Here $C_2(n, m)$ is an explicit positive constant depending only on $n$ and $m$.
\end{thm}
Gap theorems for complete submanifolds with parallel mean curvature vector field in the hyperbolic space were studied by Xu and his collaborators, where they distinguished cases of $H^2>1$ and $H^2\leq 1$. In case of $H^2>1$, Xu–Huang–Gu–He \cite{XHGH} obtained the following theorem.
\begin{thm} [\cite{XHGH}]
Let $M$ be an $n(\geq 3)$-dimensional complete submanifold with parallel mean curvature vector field in the hyperbolic space $\mathbb{H}^{n+m}$.
If $H^2>1$ and
$$\int_M|Å|^n \text{d}v<C_3(n,H),$$
then $Å=0$ and $M$ is a totally umbilical submanifold $\mathbb{S}^n(\frac{1}{\sqrt{H^2-1}})$. Here $C_3(n, H)$ is an explicit positive constant depending on $n$ and $H$.
\end{thm}
In case of $H^2\leq1$, Xu-Xu obtained:
\begin{thm} [\cite{XX}]\label{XX's thm}
Let $M$ be an $n(\geq 5)$-dimensional complete submanifold with parallel mean curvature vector field in the hyperbolic space $\mathbb{H}^{n+m}$, whose trace free second fundamental form $Å$ satisfies 
$$\limsup_{R\to\infty}\frac{\int_{B_R(q)}|Å|^2\text{d}v}{R^2}=0.$$
If the mean curvature $|H|\leq\gamma(n)$, and 
$$(\int_M|Å|^n \text{d}v)^\frac{2}{n}+\frac{2n(n-2)}{3\sqrt{n(n-1)}}H(\int_M|Å|^\frac{n}{2} \text{d}v)^\frac{2}{n}\leq C_4(n),$$ then $Å=0$, i.e. $M$ is totally umbilical. Here $B_R(q)$ denotes the geodesic ball on $M$ of radius $R$ centered at $q\in M$, $\gamma(n)=\frac{1}{5}$ for $n=5,6$, $\gamma(n)=1$ for $n\geq 7$, and $C_4(n)$ is an explicit positive constant depending only on $n$.

In particular, if $H^2<1$, then $M=\mathbb{H}^n(H^2-1)$, and if $H^2=1$,
then $M=\mathbb{R}^n$.

\end{thm}
One of  main ingredients in the proof of Theorem 1.4 is an estimate for the first eigenvalue depending on the dimension and the mean curvature due to Cheung-Leung \cite{CL}. In \cite{LW}, Lin-Wang studied gap theorem for complete noncompact submanifolds in $\mathbb{H}^n\times\mathbb{R}$ by new estimate for the first eigenvalue of complete noncompact submanifolds in $\mathbb{H}^n\times\mathbb{R}$. In particular their result implies the following interesting theorem.
\begin{thm} [\cite{LW}]
Assume that $M^n (n\geq12)$ is a complete noncompact submanifold with parallel mean curvature vector field in the hyperbolic space $\mathbb{H}^{n+m}$ and the second fundamental
form $A$ of $M$ satisfying
$$\limsup_{R\to\infty}\frac{\int_{B_R(q)}|A|^2\text{d}v}{R^2}=0,$$
where $B_R(q)$ denotes the geodesic ball on $M$ of radius $R$ centered at $q \in M$. Then there exists a positive constant $C_5(n)$ depending only
on $n$ such that if
$$\int_M|A|^n \text{d}v<C_5(n),$$
then $M$ is totally geodesic. 
\end{thm}
In this paper we continue to study gap theorems for complete submanifolds with parallel mean curvature vector field in the hyperbolic space.  We have the following slight improvement of Theorem 1.5.
\begin{thm}\label{th:1.3}
Let $M^n$ be an $n(n\ge 5)$-dimensional complete noncompact submanifold with parallel mean curvature vector field in the hyperbolic space $\mathbb{H}^{n+m}$, whose trace free second fundamental form $Å$ satisfies 
\begin{eqnarray}\label{eq:1.1}
\limsup_{R\to\infty}\frac{\int_{B_R(q)}|Å|^2\text{d}v}{R^2}=0.
\end{eqnarray}
Then there exists a constant $C_6(n)$ depending only on $n$ such that if 
\begin{eqnarray}\label{eq:1.5}
\int_M|A|^{n}\text{d}v < C_6(n),
\end{eqnarray}
then $M$ is totally geodesic. 
\end{thm}
By adapting an eigenvalue estimate for the first eigenvalue of complete noncompact submanifolds in the hyperbolic space due to Lin \cite{Lin} and modifying arguments of Xu-Xu \cite{XX} used in the proof of Theorem \ref{XX's thm}, we have
\begin{thm}\label{main thm2}
Let $M$ be an $n(\geq 5)$-dimensional complete noncompact submanifold with parallel mean curvature in the hyperbolic space $\mathbb{H}^{n+m}$, whose trace free second fundamental form $Å$ satisfies 
$$\limsup_{R\to\infty}\frac{\int_{B_R(q)}|Å|^2\text{d}v}{R^2}=0.$$
Then there exist sufficiently small constants $C_7(n), C_8(n)$ depending only on $n$ such that if
$$(\int_M|H|^n\text{d}v)^\frac{1}{n}\leq C_7(n)$$ and
$$(\int_M|Å|^n \text{d}v)^\frac{2}{n}+\frac{2n(n-2)}{3\sqrt{n(n-1)}}H(\int_M|Å|^\frac{n}{2} \text{d}v)^\frac{2}{n}\leq C_8(n),$$
then $M$ it totally geodesic.
\end{thm}

Inspired by the study of gap theorems for complete submanifolds with parallel mean curvature vector field, Li-Xu-Zhou \cite{LXZ} studied complete hypersurfaces in a Euclidean space with zero scalar curvature and they obtained the following gap theorem.
 \begin{theorem}[\cite{LXZ}]\label{LXZ}
Let $M^n(n \geq 3)$ be a locally conformally flat complete hypersurface in $\mathbb{R}^{n+1}$ with zero scalar curvature. Then there exists a positive constant $C_9(n)$ depending only on $n$ such that $M$ is a hyperplane if  $$\int_{M}|H|^{n}\text{d}v<C_9(n).$$
 \end{theorem}
By the Gauss equation,  a hypersurface in $\mathbb{R}^{n+1}$ has zero scalar curvature is equivalent to the second elementary symmetric function of its principle curvatures vanishes. Gap theorem for hypersurfaces with vanishing  second elementary symmetric function of  principle curvatures (i.e. with constant scalar curvature $n(n-1)$) in the unit sphere was obtined by Bai-Luo \cite{BL1}. Without assuming conformal flatness, instead by imposing an "elliptic" condition on the first order Newton transformation, Bai-Luo \cite{BL2} obtained 
\begin{theorem} [\cite{BL2}]
Let $(M^n, g)(n\geq3)$ be a complete  hypersurface immersed in $\mathbb{R}^{n+1}$ with zero scalar curvature. Assume that $P_1=nHg-A$ and it satisfies that
 \begin{eqnarray}
 -P_1(\nabla H,\nabla|H|)\leq-\delta|H||\nabla H|^2
 \end{eqnarray}
 for some positive constant $\delta$. Then there exists a sufficiently small number $C_{10}(n,\delta)$  depending only on dimension $n$ and $\delta$, such that if
\begin{eqnarray}
(\int_M|H|^ndv)^\frac{1}{n}<C_{10}(n,\delta),
\end{eqnarray}
then $M$ is a hyperplane.
\end{theorem}
\begin{theorem} [\cite{BL2}]
Let $(M^n, g)(n\geq3)$ be a complete  hypersurface immersed in $\mathbb{S}^{n+1}$ with constant scalar curvature $R=n(n-1)$. Assume that $P_1=nHg-A$ and it satisfies that
 \begin{eqnarray}
 -P_1(\nabla H,\nabla|H|)\leq-\delta|H||\nabla H|^2
 \end{eqnarray}
 for some positive constant $\delta$. Then there exists a sufficiently small number $C_{11}(n,\delta)$  depending only on dimension $n$ and $\delta$, such that if
\begin{eqnarray}
(\int_M|H|^ndv)^\frac{1}{n}<C_{11}(n,\delta),
\end{eqnarray}
then $M$ is totally geodesic.
\end{theorem}
The second aim of this paper is to extend the last two theorems to case of the hyperbolic space. We have
\begin{thm}\label{main theorem}
Let $M^n$ be an $n(n\ge 3)$-dimensional complete  hypersurface in the hyperbolic space $\mathbb{H}^{n+1}$ with constant scalar curvature $R=n(1-n)$. Assume that $P_1=nHg-A$ and it satisfies that 
\begin{eqnarray}\label{eq:1.3}
-P_1(\nabla H,\nabla |H|)\le-\delta |H||\nabla H|^2
\end{eqnarray}
for some positive constant $\delta > \frac{n^3}{(n-1)(n-2)}$.
 Then there exists a sufficiently small number $C_{12}(n,\delta)$  depending only on $n$ and $\delta$, such that if
\begin{eqnarray}\label{eq:1.4}
(\int_M|H|^n\text{d}v)^{\frac1 n}<C_{12}(n,\delta), 
\end{eqnarray}
then $M$ is totally geodesic.
\end{thm}
\vspace{0.1cm}

\textbf{Organization.} In the next section we will give some formulas and lemmas which will be used in  proofs of the main theorems. Theorem 1.6 and Theorem 1.7 will be proved in section 3, while Theorem \ref{main theorem} will be proved in section 4.

\section{Notations and Preliminaries}
\subsection{Basic notations}
Let $M^n$ be a complete submanifolds in the hyperbolic space $\mathbb{H}^{n+m}$ with second fundamental form $A$. We choose a local orthonormal frame $e_1,\cdots ,e_{n+m}$ in $\mathbb{H}^{n+m}$ such that $e_1,\cdots ,e_n$ are tangent to $M^n$. Let $w_1,\cdots ,w_{n+m}$ be the dual coframe. In the following we shall use the following convention on the ranges of indices:
\begin{eqnarray}
1\le A,B,C,\cdots \le n+m;\qquad  1\le i,j,k,l,\cdots \le n; \qquad n+1\le\alpha\le n+m.  \notag
\end{eqnarray}
The structure equations of $\mathbb{H}^{n+m}$ are given by
\begin{eqnarray}
dw_A=-\sum\limits_{B} w_{AB}\wedge w_B ,\qquad  w_{AB}+w_{BA}=0, \notag
\end{eqnarray}
\begin{eqnarray}
dw_{AB}=-\sum\limits_{C} w_{AC}\wedge w_{CB}+\frac 1 2 \sum\limits_{C,D} K_{ABCD}w_{C}\wedge w_{D}, \notag
\end{eqnarray}
where $K_{ABCD}=-(\delta_{AC} \delta_{BD}-\delta_{AD} \delta_{BC})$.
Restricting to $M^n$, we have $w_{\alpha}=0$. 

Since  $$0=dw_{\alpha}=-\sum\limits_{i=1}^n w_{\alpha i} \wedge w_i,$$ by Cartan's Lemma we can write
\begin{eqnarray}
w_{\alpha i}=\sum\limits_{j} h_{ij}^{\alpha}w_{j},\qquad h _{ij}^{\alpha}=h_{ji}^{\alpha}. \notag
\end{eqnarray}
The structure equations of $M^n$ are
\begin{eqnarray}
dw_i=-\sum\limits_{j} w_{ij}\wedge w_j ,\qquad  w_{ij}+w_{ji}=0, \notag
\end{eqnarray}
\begin{eqnarray}
dw_{ij}=-\sum\limits_{k} w_{ik}\wedge w_{kj}+\frac 1 2 \sum\limits_{k,l} R_{ijkl}w_{k}\wedge w_{l}, \notag
\end{eqnarray}
where the following Gauss equation holds
\begin{eqnarray}
R_{ijkl}=-(\delta_{ik} \delta_{jl}-\delta_{il} \delta_{jk})+ \sum\limits_{\alpha}(h_{ik}^{\alpha}h_{jl}^{\alpha}-h_{il}^{\alpha}h_{jk}^{\alpha}).  \label{eq:2.1}
\end{eqnarray}
Denotes by $A$, $S$, $H$ the second fundamental form, squared length of the second fundamental form and the mean curvature of $M$, respectively. We have
\begin{eqnarray*}
A=\sum\limits_{i,j,\alpha}h_{ij}^{\alpha}w_i \otimes w_j\otimes e_{\alpha}, \qquad  h_{ij}^{\alpha}=h_{ji}^{\alpha},
\end{eqnarray*}
\begin{eqnarray*}
S=\sum\limits_{i,j,\alpha}(h_{ij}^{\alpha})^2,\qquad H= \frac1 n \sqrt{ \sum\limits_{\alpha}(\sum\limits_{i}h_{ii}^{\alpha})^2}.
\end{eqnarray*}
Assume that $R$ is the scalar curvature of $M$, then by (\ref{eq:2.1}) we have
\begin{eqnarray}
R=n(1-n)+n^2 H^2 -|A|^2.  \label{eq:2.2}
\end{eqnarray}
\subsection{Submanifolds}
Let $H=|\xi|$, and when $\xi \ne 0$, we choose $e_{n+1}$ such that $e_{n+1}$ parallels with $\xi$. Then we have $$trA_{n+1}=nH$$ and $$trA_{\eta}=0,\s  n+2\le \eta \le n+m, $$ where $A_{\alpha}= \left\{ h_{ij}^{\alpha}\right\}_{n\times n}.$  Denote by $Å_\alpha$ the trace free part of $A_\alpha$.

By \cite{XX}, we have the following Simons' formula for submanifolds with parallel mean curvature vector field in $\mathbb{H}^{n+m}$.
\begin{align} \label{eq:2.6}
\frac1 2 \Delta|Å|^2 &=|\nabla Å|^2-n|Å|^2+nH^2|Å|^2+nH\sum\limits_{\alpha}tr(Å_{n+1}Å_{\alpha}^{2})   \notag\\
& -\sum\limits_{\alpha,\eta}tr([Å_{\alpha},Å_{\eta}]^2)-\sum\limits_{\alpha,\eta}[tr(Å_{\alpha}Å_{\eta})]^2  \notag \\
&\ge |\nabla Å|^2-n|Å|^2+nH\sum\limits_{\alpha}tr(Å_{n+1}Å_{\alpha}^{2})  \notag \\
&-\sum\limits_{\alpha,\eta}tr([Å_{\alpha},Å_{\eta}]^2)-\sum\limits_{\alpha,\eta}[tr(Å_{\alpha}Å_{\eta})]^2.
\end{align}
\begin{lem} [\cite{HS}]\label{lem:2.2}
Assume that $M^n$ is a complete submanifold in $N^{n+m}$ with nonpositive sectional curvatures, then 
$$(\int_M g^{\frac{n}{n-1}}\text{d}v)^\frac{n-1}{n}\le D(n)\int_M(|\nabla g|+n|H|g)\text{d}v,$$
holds for any $g\in C_{0}^{1}(M)$, where $D(n)=\frac {2^n(1+n)^{\frac {n+1} n}}{(n-1)(\sigma_n)^{\frac 1 n}}$ and $\sigma_n$ is volume of the unit ball in $\mathbb{R}^n$.
\end{lem}
From the above lemma we can derive the following
\begin{lem} \label{lem:2.3}
Let $M^n$ be a complete submanifold immersed in $\mathbb{H}^{n+m}$. Suppose that $D(n)n||H||_n<1$ where
$D(n)$ is the constant in Lemma \ref{lem:2.2}. Then for any $f\in C_{0}^{1}(M)$, we have
$$(\int_M f^{\frac{2n}{n-2}}\text{d}v)^\frac{n-2}{n}\le C_{13}(n)\int_M |\nabla f|^2\text{d}v,$$
where $C_{13}(n)=(\frac{D(n)}{1-D(n)n||H||_n} \frac{2(n-1)}{n-2})^2$.  
\end{lem}
\begin{proof} For a function $h$ as in Lemma \ref{lem:2.2}, we have
$$(\int_M h^{\frac{n}{n-1}}\text{d}v)^\frac{n-1}{n}\le D(n)\int_M(|\nabla h|+n|H|h)\text{d}v,$$
Let $h=f^{\frac{2(n-1)}{n-2}}$, by Hölder's inequality, we have
\begin{align}
(\int_M f^{\frac{2n}{n-2}}\text{d}v)^\frac{n-1}{n} &\le D(n)\frac{2(n-1)}{n-2}\int_M f^{\frac{n}{n-2}}|\nabla f|\text{d}v+D(n)n\int_M |H|f^{\frac{2(n-1)}{n-2}}\text{d}v \notag \\
&\le D(n) \frac{2(n-1)}{n-2} (\int_M f^{\frac{2n}{n-2}}\text{d}v)^{\frac {1}{2}} (\int_M |\nabla f|^{2}\text{d}v)^{\frac 1 2}\notag \\
&+ nD(n)(\int_M f^{\frac{2n}{n-2}}\text{d}v)^{\frac {n-1}{n}} (\int_M |H|^{n}\text{d}v)^{\frac 1 n}. \notag
\end{align}
Hence
$$(\int_M f^{\frac{2n}{n-2}}\text{d}v)^\frac{n-2}{n}\le (\frac{D(n)}{1-D(n)n||H||_n} \frac{2(n-1)}{n-2})^2 \int_M |\nabla f|^2\text{d}v,$$
which implies the conclusion of this lemma.
\end{proof}

The following lower bound estimate for the first eigenvalue of complete noncompact submanifolds in the hyperbolic space will be very important in proofs of our main theorems.
\begin{lem}[\cite{Lin}] \label{lem:2.4}
 Let $M^n$ be a complete noncompact submanifold in $\mathbb{H}^{n+m}$ and $\lambda_1 (M)$ the first eigenvalue of $M$ which is defined by
$$\lambda_1(M)=\underset{f\ne 0,f\in  \mathring{H}_{1}^{2}(M)}{inf} \frac{\int_M |\nabla f|^2\text{d}v}{\int_M f^2\text{d}v}.$$
If $D(n)n||H||_n<1$, we have
\begin{align} \label{eq:2.7}
\lambda_1(M) \ge \frac{(n-1)^2(1-D(n)n||H||_n)^2}{4}>0,
\end{align}
where $D(n)=\frac {2^n(1+n)^{\frac {n+1} n}}{(n-1)(\sigma_n)^{\frac 1 n}}$ and $\sigma_n$ is volume of the unit ball in $\mathbb{R}^n$.
\end{lem}
\begin{lem} [\cite{XX}] \label{lem:2.5}
 Let $M^n$ be $n$-dimensional immersed submanifold with parallel mean curvature in $\mathbb{H}^{n+m},$ then $$|\nabla Å|^2-|\nabla|Å||^2 \ge \frac 2 n |\nabla|Å||^2.$$
\end{lem}
\begin{lem} [\cite{San}] \label{lem:2.6}
Let $a_1,\cdots ,a_n$ and $b_1,\cdots,b_n$ be real numbers satisfying
$\sum\limits_{i}a_i=\sum\limits_i b_i=0, \sum\limits_i a_{i}^{2}=a$ and $\sum\limits_i b_{i}^{2}=b.$ Then we have $$|\sum\limits_i a_i b_{i}^{2}|\le d(n)\sqrt{a}b,$$
where $d(n)=\frac{n-2}{\sqrt{n(n-1)}}.$
\end{lem}
\subsection{Hypersurfaces}
For a smooth function $f$ on a hypersurface $M^n$ in $\mathbb{H}^{n+1}$, we introduce an operator $\Box$ due to Cheng-Yau \cite{CY} by 
\begin{eqnarray*}
\Box f\equiv (nH\delta_{ij}-h_{ij})f_{ij},
\end{eqnarray*}
and we denote $(P_1)_{ij}=nH\delta_{ij}-h_{ij}$ in the rest of this paper. Note that $P_1=nHg-A$ is the first order Newton transformation.
Then (see \cite{ACC, Li})
\begin{eqnarray}
\Box H=\frac 1 n |\nabla A|^2 - n|\nabla H|^2-|A|^2+nH^2+HtrA^3-\frac 1 n |A|^4.  \label{eq:2.3}
\end{eqnarray}
It is easy to see that $$|trA^3|\le |A|^3.$$
Therefore we obtain $$\Box H \ge \frac 1 n |\nabla A|^2 - n|\nabla H|^2-|A|^2+nH^2-|H||A|^3-\frac 1 n |A|^4.$$
If the scalar curvature $R=n(1-n)$, then by (\ref{eq:2.2}) we see that     $$n^2 H^2=|A|^2,$$
and hence 
\begin{eqnarray}
\Box H \ge \frac 1 n |\nabla A|^2 - n|\nabla H|^2+n(1-n)H^2-2n^3H^4.\label{eq:2.4}
\end{eqnarray}
\begin{lem}\label{lem:2.1}
Let $M^n$ is a hypersurface in the hyperbolic space $\mathbb{H}^{n+1}$ with constant scalar curvature $R=n(1-n)$, then $$\frac 1 n |\nabla A|^2\ge n|\nabla H|^2.$$
\end{lem}
\begin{proof} Since the scalar curvature $R=n(1-n)$,  by (\ref{eq:2.2}) we see that $$A\nabla A=n^2H\nabla H,$$ which implies that $$|A\nabla A|^2=n^4|H|^2|\nabla H|^2.$$
By Cauchy-Schwarz's inequality
$$|A \nabla A|^2\le|A|^2|\nabla A|^2.$$
Therefore
$$|A|^2|\nabla A|^2\ge n^4|H|^2|\nabla H|^2,$$
which implies that $$H^2(|\nabla A|^2-n^2|\nabla H|^2)\ge 0.$$
Therefore, for an arbitrary point $p\in M$,
\\if $H(p)\ne 0$, we have $$|\nabla A|^2\ge n^2|\nabla H|^2$$ at $p$, and
\\if $H(p)=0$, we have $|A|(p)=0,$ hence $$|\nabla A|^2=n^2|\nabla H|^2$$ at $p$ by (\ref{eq:2.3}).
\end{proof}

By the above lemma and (\ref{eq:2.4}) we finally arrive at
\begin{eqnarray}
\Box H \ge n(1-n)H^2-2n^3H^4.\label{eq:2.5}
\end{eqnarray}

\section{Proofs of Theorem \ref{th:1.3} and Theorem \ref{main thm2}}
\textbf{Proof of Theorem \ref{th:1.3}.} Let $M^n$ be an $n(n\ge 5)$-dimensional complete submanifold with parallel mean curvature in the hyperbolic space $\mathbb{H}^{n+m}$. According to \cite{LL}, we have $$\sum\limits_{\alpha,\eta}tr([Å_{\alpha},Å_{\eta}]^2)+\sum\limits_{\alpha,\eta}[tr(Å_{\alpha}Å_{\eta})]^2 \le b(m)|Å|^4 ,$$
where $$b(m)=1+\frac 1 2 sgn(m-1).$$
Recalling that $$\Delta|Å|^2=2|Å|\Delta|Å|+2|\nabla|Å||^2,$$ together with Lemma \ref{lem:2.5}, the above inequality imply that
\begin{align} \label{eq:4.1}
b(m)|Å|^4 &\ge |\nabla Å|^2-n|Å|^2+nH\sum\limits_{\alpha}tr(Å_{n+1}Å_{\alpha}^{2})-\frac1 2 \Delta|Å|^2 \notag \\
&=-|Å|\Delta|Å|-|\nabla|Å||^2+|\nabla Å|^2-n|Å|^2+nH\sum\limits_{\alpha}tr(Å_{n+1}Å_{\alpha}^{2})  \notag \\
&\ge -|Å|\Delta|Å|+\frac2 n |\nabla|Å||^2-n|Å|^2+nH\sum\limits_{\alpha}tr(Å_{n+1}Å_{\alpha}^{2}). 
\end{align}
Let $\phi$ be a function in $C_{0}^{\infty}(M).$  Multiplying (\ref{eq:4.1}) by $\phi^2$ and integrating by parts on $M$, we have
\begin{align} \label{eq:4.2}
b(m)\int_M \phi^2|Å|^4\text{d}v &\ge -\int_M \phi^2|Å|\Delta|Å|\text{d}v+\frac2 n \int_M \phi^2|\nabla|Å||^2\text{d}v \notag \\
&-n\int_M \phi^2|Å|^2\text{d}v+nH\int_M \phi^2 \sum\limits_{\alpha}tr(Å_{n+1}Å_{\alpha}^{2})\text{d}v \notag \\
&=\int_M \phi^2|\nabla|Å||^2\text{d}v+2\int_M \phi |Å|\langle \nabla|Å|,\nabla \phi \rangle \text{d}v +\frac2 n \int_M \phi^2|\nabla|Å||^2\text{d}v\notag \\
& -n\int_M \phi^2|Å|^2\text{d}v+nH\int_M \phi^2 \sum\limits_{\alpha}tr(Å_{n+1}Å_{\alpha}^{2})\text{d}v.
\end{align}
For a fixed $e_{\alpha}$, we choose an orthonormal local frame field $e_1,\cdots,e_n$ to diagonalize the matrix $\left\{ h_{ij}^{\alpha}\right\}$, i.e. $h_{ij}^{\alpha}=\lambda_{i}^{\alpha}\delta_{ij},$ $1\le i,j \le n.$ By Lemma 2.6, the last term of the above inequality is estimated as follows.
\begin{align}
|tr(Å_{n+1}Å_{\alpha}^{2})|&=|\sum\limits_i[Å_{n+1}]_{ii}([Å_{\alpha}]_{ii})^2| \notag \\
&\le d(n)|Å_{n+1}||Å_{\alpha}|^2.  \notag
\end{align}
Combining it with $$nH\le \sqrt{n}|A|,$$ we get
\begin{align}
|nH\int_M \phi^2 \sum\limits_{\alpha}tr(Å_{n+1}Å_{\alpha}^{2})\text{d}v|&\le nd(n)H\int_M \phi^2 |Å_{n+1}||Å|^2\text{d}v  \notag \\
&\le \sqrt{n}d(n)\int_M \phi^2|A|^2|Å|^{2}\text{d}v \notag \\
&\le \sqrt{n}d(n)(\int_M |A|^{n}\text{d}v)^{\frac 2 n}(\int_M |\phi |Å||^{\frac {2n} {n-2}}\text{d}v)^{\frac {n-2} n}. \notag 
\end{align}
By the Hölder's inequality, we get
$$b(m)\int_M \phi^2|Å|^4\text{d}v \le b(m)(\int_M |A|^{n}\text{d}v)^{\frac 2 n}(\int_M |\phi |Å||^{\frac {2n} {n-2}}\text{d}v)^{\frac {n-2} n}.$$
Set
\begin{align}
\Gamma=(b(m)+\sqrt{n}d(n))C_6(n)^{\frac 2 n}.\notag
\end{align}
Assume that  $n^\frac{1}{2}C_6(n)^\frac{1}{n}D(n)<1$, then since $$nH\le \sqrt{n}|A|,$$ we have $nD(n)\|H\|_n<1$. Using Lemma \ref{lem:2.3}, we get an upper bound for $(\int_M |\phi |Å||^{\frac {2n} {n-2}}\text{d}v)^{\frac {n-2} n}$ as follows
$$(\int_M |\phi |Å||^{\frac {2n} {n-2}}\text{d}v)^{\frac {n-2} n}\le C_{13}(n)\int_M |\nabla(\phi |Å|)|^2\text{d}v.$$
Substituting the above inequality into (\ref{eq:4.2}), we obtain
\begin{align} \label{eq:4.3}
(1+\frac 2 n)\int_M \phi^2|\nabla|Å||^2\text{d}v &\le b(m)\int_M \phi^2 |Å|^4\text{d}v-nH\int_M \phi^2 \sum\limits_{\alpha}tr(Å_{n+1}Å_{\alpha}^{2})\text{d}v \notag \\
&-2\int_M \phi |Å|\langle \nabla|Å|,\nabla \phi \rangle \text{d}v+n\int_M \phi^2|Å|^2\text{d}v \notag \\
&\le C_{13}(n) \Gamma \int_M|\nabla(\phi |Å|)|^2\text{d}v+n\int_M |Å|^2 \phi^2 \text{d}v \notag \\
&-2\int_M \phi |Å|\langle \nabla|Å|,\nabla \phi \rangle \text{d}v.
\end{align}
By  Lemma \ref{lem:2.4} we get
$$\int_M |Å|^2\phi^2\text{d}v \le \frac {4}{(n-1)^2(1-D(n)n||H||_n)^2}\int_M |\nabla(\phi|Å|)|^2\text{d}v. $$
This together with (\ref{eq:4.3}) imply that
\begin{align} \label{eq:4.4}
(1+\frac 2 n)\int_M \phi^2|\nabla|Å||^2\text{d}v \le L \int_M|\nabla(\phi |Å|)|^2\text{d}v-2\int_M \phi |Å|\langle \nabla|Å|,\nabla \phi \rangle \text{d}v,
\end{align}
where $L=\Gamma(\frac{D(n)}{1-D(n)n||H||_n} \frac{2(n-1)}{n-2})^2+\frac {4n}{(n-1)^2(1-D(n)n||H||_n)^2}.$
We further choose the constant $C_6(n) (n\geq5)$ in Theorem \ref{th:1.3} sufficiently small such that $$1+\frac{2}{n}-L>0.$$ 
Then
\begin{align} \label{eq:4.5}
&(1+\frac 2 n-L)\int_M \phi^2|\nabla|Å||^2\text{d}v \notag \\
&\le L \int_M|\nabla \phi|^2 |Å|^2\text{d}v+2(L-1)\int_M \phi |Å|\langle \nabla|Å|,\nabla \phi \rangle \text{d}v.
\end{align}
Using Cauchy's inequality for arbitrary $\theta>0$, we have
\begin{align} 
&2(L-1)\int_M \phi |Å|\langle \nabla|Å|,\nabla \phi \rangle \text{d}v
\notag \\
&\le |L-1|\theta \int_M \phi^2|\nabla|Å||^2\text{d}v+\frac{|L-1|}{\theta}\int_M|\nabla \phi|^2|Å|^2\text{d}v. \notag
\end{align}
When $L\ne1$, we choose $\nu$ such that $0<\nu \le 1+\frac2 n -L$ and
setting $\theta=\frac{\nu}{2|L-1|},$ we get
$$\int_M \phi^2|\nabla|Å||^2\text{d}v \le \rho \int_M|\nabla \phi|^2 |Å|^2\text{d}v,$$
where $\rho=\frac{L+\frac{2|L-1|^2}{\nu}}{1+\frac{2}{n}-L-\frac{\nu}{2}}.$

Fix a point $q\in M$, we choose $\phi$ to be a cut-off function with the properties
\begin{align*}
\phi=\begin{cases}
1, \qquad  \text{on } B_R(q),    \\
0,   \qquad  \text{on } M \setminus B_{2R}(q),  \\ 
0\le \phi \le1,|\nabla \phi|\le \frac{C_{14}(n)}{R},  \qquad  \text{on } B_{2R}(q) \setminus B_R(q). 
\end{cases}
\end{align*}
Substituting it into the above inequality, we get
\begin{align}\label{eq:4.6}
&\int_{B_R(q)} |\nabla|Å||^2\text{d}v \nonumber
\\ &=\int_M \phi^2|\nabla|Å||^2\text{d}v  \notag \\
&\le \rho \int_{B_{2R}(q)}|\nabla \phi|^2 |Å|^2\text{d}v \notag \\
&\le \frac{C_{14}(n)^2\rho}{R^2}\int_{B_{2R}(q)}|Å|^2\text{d}v.
\end{align}
Letting $R\rightarrow \infty$ in the above inequality we have
$$\int_M |\nabla|Å||^2\text{d}v=0.$$
Therefore, $|Å|=d=constant$. If $d\ne 0$, we know from (\ref{eq:1.1}) that
$$\limsup_{R\to \infty}\frac{Vol[B_R(q)]}{R^2}=0.$$
It then follows from \cite{CY1}  that $\lambda_1(M)=0$ which contradicts with (\ref{eq:2.7}). Hence $|Å| = 0,$ i.e. $M$ is umbilical. Furthermore $M$ is totally geodesic, since it is noncompact. 
\\
\\\textbf{Proof of Theorem \ref{main thm2}.} Note that
\begin{align*}
|nH\int_M \phi^2 \sum\limits_{\alpha}tr(Å_{n+1}Å_{\alpha}^{2})\text{d}v|&\le nd(n)H\int_M \phi^2 |Å_{n+1}||Å|^2\text{d}v
\\&\le nd(n)H\int_M \phi^2|Å|^3\text{d}v
\\&\le nd(n)H(\int_M|Å|^\frac{n}{2}\text{d}v)^\frac{2}{n}(\int_M |\phi |Å||^{\frac {2n} {n-2}}\text{d}v)^{\frac {n-2} n}
\end{align*}
and 
$$b(m)\int_M \phi^2|Å|^4\text{d}v \le b(m)(\int_M |Å|^{n}\text{d}v)^{\frac 2 n}(\int_M |\phi |Å||^{\frac {2n} {n-2}}\text{d}v)^{\frac {n-2} n}.$$
Then if $nC_7(n)D(n)<1$, by Lemma \ref{lem:2.3} we have
$$(\int_M |\phi |Å||^{\frac {2n} {n-2}}\text{d}v)^{\frac {n-2} n}\le C_{13}(n)\int_M |\nabla(\phi |Å|)|^2\text{d}v.$$
Substituting the above inequality into (\ref{eq:4.2}), we obtain
\begin{align} \label{eq:4.3'}
(1+\frac 2 n)\int_M \phi^2|\nabla|Å||^2\text{d}v &\le b(m)\int_M \phi^2 |Å|^4\text{d}v-nH\int_M \phi^2 \sum\limits_{\alpha}tr(Å_{n+1}Å_{\alpha}^{2})\text{d}v \notag \\
&-2\int_M \phi |Å|\langle \nabla|Å|,\nabla \phi \rangle \text{d}v+n\int_M \phi^2|Å|^2\text{d}v \notag \\
&\le C_{13}(n) b(m)C_8(n) \int_M|\nabla(\phi |Å|)|^2\text{d}v+n\int_M |Å|^2 \phi^2 \text{d}v \notag \\
&-2\int_M \phi |Å|\langle \nabla|Å|,\nabla \phi \rangle \text{d}v.
\end{align}
By  Lemma \ref{lem:2.4} we get
$$\int_M |Å|^2\phi^2\text{d}v \le \frac {4}{(n-1)^2(1-D(n)n||H||_n)^2}\int_M |\nabla(\phi|Å|)|^2\text{d}v. $$
This together with (\ref{eq:4.3'}) imply that
\begin{align} \label{eq:4.4'}
(1+\frac 2 n)\int_M \phi^2|\nabla|Å||^2\text{d}v \le L \int_M|\nabla(\phi |Å|)|^2\text{d}v-2\int_M \phi |Å|\langle \nabla|Å|,\nabla \phi \rangle \text{d}v,
\end{align}
where $L=b(m)C_8(n)(\frac{D(n)}{1-D(n)n||H||_n} \frac{2(n-1)}{n-2})^2+\frac {4n}{(n-1)^2(1-D(n)n||H||_n)^2}.$ 

Note that if $n\geq5$ and $C_7(n), C_8(n)$ are sufficiently small, we have $1+\frac{2}{n}-L>0$ . Then following arguments used in last part of the proof of Theorem \ref{th:1.3} we can prove that $M$ is totally geodesic.

\section{Proof of Theorem \ref{main theorem}}
Proof. Assume that $M^n$ is a complete hypersurface in $\mathbb{H}^{n+1}$ with scalar curvature $R=n(1-n)$. Let $B_r$ be a geodesic ball centered at a fixed point of $M^n$ with radius $r$, and  choose $\psi$ to be a cut-off function with the properties
\begin{align*}
\gamma=\begin{cases}
1, \qquad  \text{on } B_R,    \\
0,   \qquad  \text{on } M \setminus B_{2R},  \\ 
0\le \psi \le1,|\nabla \psi|\le \frac{C_{15}(n)}{R},  \qquad  \text{on } B_{2R} \setminus B_R. 
\end{cases}
\end{align*}


We denote by $f=|H|$. Multiplying $\psi^2 f^{2k-2}(k\ge 1)$ on both sides of inequality (\ref{eq:2.5}) and integrating it over $M$ we obtain
\begin{eqnarray}
\int_M \psi^2 f^{2k-2}\Box H\text{d}v\ge n(1-n)\int_M \psi^2 f^{2k}\text{d}v-2 n^3\int_M \psi^2 f^{2k+2}\text{d}v.\label{eq:3.1}
\end{eqnarray}
By integrating by parts and Cauchy-Schwarz's inequality, we have
\begin{align}
&\int_M \psi^2 f^{2k-2}\Box H\text{d}v  \notag\\
& = 2\int_M \psi f^{2k-2}(-P_1)(\nabla \psi,\nabla H)\text{d}v +2(k-1) \int_M \psi^2 f^{2k-3}(-P_1)(\nabla H,\nabla |H|)\text{d}v\notag  \\
& +\int_M \psi^2 f^{2k-2}\sum\limits_{i,j}(-P_1)_{ijj}H_i\text{d}v   \notag  \\
&=2\int_M \psi f^{2k-2}(-P_1)(\nabla \psi,\nabla H)\text{d}v+2(k-1) \int_M \psi^2 f^{2k-3}(-P_1)(\nabla H,\nabla |H|)\text{d}v  \notag \\
&\le 2\int_M \psi f^{2k-2}(-P_1)(\nabla \psi,\nabla H)\text{d}v-2\delta(k-1) \int_M \psi^2 f^{2k-2}|\nabla H|^2\text{d}v  \notag \\
&\le 2\int_M \psi f^{2k-2}|P_1||\nabla \psi||\nabla H|\text{d}v-2\delta(k-1) \int_M \psi^2 f^{2k-2}|\nabla H|^2\text{d}v , \notag 
\end{align}
where in the second equality we used
\begin{eqnarray*}
&&\sum\limits_{i,j}(-P_1)_{ijj}H_i
\\&=&\sum\limits_{i,j}(h_{ijj}H_i-nH_jH_i\delta_{ij})
\\&=&\sum\limits_{i,j}(h_{jji}H_i-nH_jH_i\delta_{ij})
\\&=&0.
\end{eqnarray*}
Note that 
\begin{eqnarray*}
|P_1|^2&=&|A|^2-2n^2H^2+n^3H^2
\\&=&n^2(n-1)H^2
\\&=&n^2(n-1)f^2.
\end{eqnarray*}
By Cauchy-Schwarz's inequality for any $\tau$ such that $0<\tau<2$, we have
\begin{align}
&\int_M \psi^2 f^{2k-2}\Box H\text{d}v  \notag \\
&\le 2n\sqrt{n-1}\int_M \psi f^{2k-1}|\nabla \psi||\nabla H|\text{d}v-2\delta(k-1) \int_M \psi^2 f^{2k-2}|\nabla H|^2\text{d}v  \notag\\
&\le \tau \delta (k-1)\int_M \psi^2 f^{2k-2}|\nabla H|^2\text{d}v+\frac{n^2 (n-1)}{\tau \delta (k-1)}\int_M |\nabla \psi|^2 f^{2k}\text{d}v  \notag\\
&-2\delta(k-1) \int_M \psi^2 f^{2k-2}|\nabla H|^2\text{d}v\notag\\
&\le (\tau-2) \delta (k-1)\int_M \psi^2 f^{2k-2}|\nabla f|^2\text{d}v+\frac{n^2 (n-1)}{\tau \delta (k-1)}\int_M |\nabla \psi|^2 f^{2k}\text{d}v  \notag\\
&=(\tau-2) \delta \frac{k-1}{k^2}\int_M \psi^2 |\nabla f^k|^2 \text{d}v+\frac{n^2 (n-1)}{\tau \delta (k-1)}\int_M |\nabla \psi|^2 f^{2k}\text{d}v. \notag
\end{align}
By the above inequality and (\ref{eq:3.1}) we finally arrive at
\begin{align}
&(2-\tau) \delta \frac{k-1}{k^2}\int_M \psi^2 |\nabla f^k|^2 \text{d}v+n(1-n)\int_M \psi^2f^{2k}\text{d}v\notag\\
&\le 2n^3\int_M\psi^2f^{2k+2}\text{d}v+\frac{n^2 (n-1)}{\tau \delta (k-1)}\int_M |\nabla \psi|^2 f^{2k}\text{d}v. \notag
\end{align}
Let $k=\frac n 2$ in the above inequality we obtain
\begin{align}
&(2-\tau) \delta \frac{2n-4}{n^2}\int_M \psi^2 |\nabla f^{\frac n 2}|^2 \text{d}v+n(1-n)\int_M \psi^2f^{n}\text{d}v\notag\\
&\le 2n^3\int_M\psi^2f^{n+2}\text{d}v+\frac{2n^2 (n-1)}{\tau \delta (n-2)}\int_M |\nabla \psi|^2 f^{n}\text{d}v. \label{eq:3.2}
\end{align}
Assume that $D(n)nC_{12}(n,\delta) <1$. By Lemma \ref{lem:2.3} we have
\begin{align}
\int_M \psi^2 f^{n+2}\text{d}v &=\int_M (\psi f^{\frac n 2})^2f^2 \text{d}v \notag\\
&\le (\int_M f^{n}\text{d}v)^{\frac 2 n}(\int_M ( \psi f^{\frac n 2})^{\frac{2n}{n-2}}\text{d}v)^{\frac{n-2}{n}}\notag \\
&\le C_{12}(n,\delta)^2 C_{13}(n)\int_M |\nabla (\psi f^{\frac n 2})|^2\text{d}v   \notag \\
&\le 2C_{12}(n,\delta)^2 C_{13}(n) (\int_M |\nabla \psi|^2 f^{n}\text{d}v+\int_M\psi^2 |\nabla f^{\frac n 2}|^2 \text{d}v). \notag
\end{align}
Combing the above two inequalities we get
\begin{align}
E_1\int_M \psi^2 |\nabla f^{\frac n 2}|^2 \text{d}v+E_2\int_M \psi^2f^{n}\text{d}v\le E_3\int_M |\nabla \psi|^2 f^{n}\text{d}v, \label{eq:3.4}
\end{align}
where $$E_1=(2-\tau) \delta \frac {2n-4}{n^2}-4n^3C_{12}(n,\delta)^2 C_{13}(n),$$
$$E_2=(1-n)n,$$
$$E_3=\frac {2n^2(n-1)}{\tau \delta(n-2)}+4n^3C_{12}(n,\delta)^2 C_{13}(n).$$
Using Lemma \ref{lem:2.4} and Cauchy-Schwarz's  inequality we have for any $\kappa>0$
\begin{align}
\int_M \psi^2 f^n\text{d}v &\le \frac{4}{(n-1)^2(1-nD(n)||H||_n)^2}\int_M |\nabla (\psi f^{\frac{n}{2}})|^2\text{d}v  \notag  \\
&\le \frac{4}{(n-1)^2(1-nD(n)||H||_n)^2}[(1+\frac {1} {\kappa})\int_M |\nabla \psi|^2 f^{n}\text{d}v+(1+\kappa)\int_M \psi^2|\nabla f^{\frac n 2}|^2 \text{d}v].  \notag 
\end{align}
Substituting the above inequality into (\ref{eq:3.4}) yields
$$F_1\int_M \psi^2|\nabla f^{\frac n 2}|^2\text{d}v \le F_2\int_M |\nabla \psi|^2 f^n\text{d}v, $$
where
\begin{align}
F_1&=(1-n)n \frac{4(1+\kappa)}{(n-1)^2(1-nD(n)||H||_n)^2}+(2-\tau) \delta \frac {2n-4}{n^2}-4n^3C_{12}(n,\delta)^2 C_{13}(n), \notag
\end{align}
\begin{align}
F_2&=-[(1-n)n \frac{4(1+\frac {1} {\kappa})}{(n-1)^2(1-nD(n)||H||_n)^2 }]+\frac {2n^2(n-1)}{\tau \delta(n-2)}+4n^3C_{12}(n,\delta)^2 C_{13}(n). \notag
\end{align}
Since $\delta>\frac{n^3}{(n-1)(n-2)}$, when $\tau, \kappa$ and $C_{12}(n,\delta)$ are sufficiently small we have $F_1>0$.
Then there exists a constant $C_{16}(n,\delta)$ depending only on $n$ and $\delta$ such that
$$\int_M \psi^2 |\nabla |H|^{\frac n 2}|^2\text{d}v \le \frac{C_{16}(n,\delta)}{R^2}\int_M |H|^n\text{d}v.$$
Letting $R\rightarrow +\infty$ we get
$$\int_M |\nabla |H|^{\frac n 2}|^2\text{d}v=0, $$
which implies $H=0$ by Lemma \ref{lem:2.4}. Then by (\ref{eq:2.2}) we see that $A=0$, i.e. $M$ is totally geodesic.

\section*{Acknowledgments}
This work is supported by the NSF of China (Grant No. 12271069) and the NSF of Chongqing (Grant No. CSTB2024NSCQ-LZX0093). The authors would like to thank Professor Hezi Lin for his interests and stimulating discussions on this subject.
{}
\vspace{1cm}\sc

Yong Luo

Mathematical Science Research Center of Mathematics,

Chongqing University of Technology,

Chongqing, 400054, China

{\tt yongluo-math@cqut.edu.cn}

\vspace{1cm}\sc
Jianling Liu

Mathematical Science Research Center of Mathematics,

Chongqing University of Technology,

Chongqing, 400054, China

{\tt 18315037508@163.com}


\begin{thebibliography}{2}
\bibitem{ACC}H. Alencar, M. do Carmo and A. G. Colares, Stable hypersurfaces with constant scalar curvature, {\em Math. Z.} {\bf 213} (1993), 117--131.
\bibitem{BL1} J. C. Bai and Y. Luo,  On complete hypersurfaces with constant scalar curvature $n(n-1)$ in the unit sphere, {\em Kodai Math. J.} {\bf46} (2023), no. 1, 51--61.
\bibitem{BL2} J. C. Bai and Y. Luo, Remarks on gap theorems for complete hypersurfaces with constant scalar curvature, {\em J. Math. Study} {\bf56} (2023), no. 3, 279--290.
\bibitem{CY1}S. Y. Cheng and S. T. Yau, Differential equations on Riemannian manifolds and their geometric applications, {\em Comm. Pure Appl. Math.} {\bf 28} (1975), 333--354.
\bibitem{CY}S. Y. Cheng and S. T. Yau, Hypersurfaces with constant scalar curvature, {\em Math. Ann.} {\bf 225} (1977), 195--204.
\bibitem{CDK} S. S. Chern, M. do Carmo and S. Kobayashi, Minimal submanifolds of a sphere with second fundamental form of constant length, in: Functional Analysis and Related Fields, {\em Springer-Verlag}, Berlin, 1970, pp. 59--75.
\bibitem{CL} L. F. Cheung and P. F. Leung, Eigenvalue estimates for submanifolds with bounded mean curvature in the hyperbolic space,
{\em Math. Z.} {\bf236} (2001), 525--530.
\bibitem{HS} D. Hoffman and J. Spruck, Sobolev and isoperimetric inequalities for Riemannian submanifolds, {\em Comm. Pure Appl. Math.} {\bf 27} (1974), 715--727.
\bibitem{Law}B. Lawson, Local rigidity theorems for minimal hypersurfaces, {\em Ann. Math.} {\bf 89} (1969), 187--197.
\bibitem{LL} A. M. Li and J. M. Li, An intrinsic rigidity theorem for minimal submanifolds in a sphere, {\em Arch. Math.} {\bf 58} (1992), 582--594.
\bibitem{Li}H. Z. Li, Hypersurfaces with constant scalar curvature in space forms, {\em Math. Ann.} {\bf 305} (1996), 665--672.
\bibitem{LXZ} Y. W. Li, X. W. Xu and J. R. Zhou, The complete hyper-surfaces with zero scalar curvature in $\mathbb{R}^{n+1}$, {\em Ann. Global Anal. Geom.} {\bf44} (2013), no.4, 401--416.
\bibitem{Lin} H. Z. Lin, Eigenvalue estimate and gap theorems for submanifolds in the hyperbolic space, {\em Nonlinear Analysis: Theory, Methods Applications} {\bf 148} (2017), 126--137.
\bibitem{LW} H. Z. Lin and X. S. Wang, Gap theorems for submanifolds in $\mathbb{H}^n\times \mathbb{R},$ {\em J. Geom. Phys.} {\bf160} (2021), Paper No. 103998, 11 pp.
\bibitem{LX} J. M. Lin and C.Y. Xia, Global pinching theorem for even dimensional minimal submanifolds in a unit sphere, {\em Math. Z.} {\bf201} (1989), 381--389.
\bibitem{Ni} L. Ni, Gap theorems for minimal submanifolds in $\mathbb{R}^{n+1}$, {\em Comm. Anal. Geom.} {\bf9} (2001), no.3, 641--656.
\bibitem{San}W. Santos, Submanifolds with parallel mean curvature vector in spheres, {\em Tohoku Math. J.} {\bf 46} (1994), 403--415.
\bibitem{Shen} C. L. Shen, A global pinching theorem for minimal hypersurfaces in sphere, {\em Proc. Am. Math. Soc.} {\bf105} (1989), 192--198.
\bibitem{Sim} J. Simons, Minimal varieties in Riemannian manifolds, {\em Ann. of Math.} {\bf 88} (1968), 62--105.
\bibitem{XiaW} C. Y. Xia and Q. L. Wang, Gap theorems for minimal submanifolds of a hyperbolic space, {\em J. Math. Anal. Appl.} {\bf 436} (2016), 983--989.
\bibitem{Xu} H. W. Xu, $L^\frac{n}{2}$-pinching theorems for submanifolds with parallel mean curvature in a sphere, {\em J. Math. Soc. Jpn.} {\bf46} (1994), 503--515.
\bibitem{XuG} H. W. Xu and J. R. Gu, A general gap theorem for submanifolds with parallel mean curvature in $\mathbb{R}^{n+p}$, {\em Commun. Anal. Geom.} {\bf15} (2007) 175--194.
\bibitem{XHGH} H. W. Xu, F. Huang, J. R. Gu and M. Y. He, $L^\frac{n}{2}$ pinching theorem for submanifolds with parallel mean curvature in $\mathbb{H}^{n+p}(-1)$, {\em Pure Appl. Math. Q.} {\bf8} (2012), 1097--1115.
\bibitem{XX} Z. Y. Xu and H. W. Xu, A gap theorem for complete submanifolds with parallel mean
curvature in the hyperbolic space, {\em J. Math. Anal. Appl.} {\bf494} (2021), 124549.
\bibitem{Yun} G. Yun, Total scalar curvature and $L^2$ harmonic 1-forms on a minimal hypersurface in Euclidean space, {\em Geom. Dedicata} {\bf89} (2002), 135--141.

  

    
  
\end{thebibliography}
\end{document}